\documentclass{scrartcl}  
\usepackage[utf8]{inputenc}

\usepackage{amssymb}
\usepackage{amsthm}
\usepackage{amsmath}

\usepackage[colorinlistoftodos,prependcaption]{todonotes}

\usepackage[pdfauthor={},
	colorlinks=true,
	linkcolor=black,
	citecolor=black,
	urlcolor=black,
	filecolor=black,
	pdftitle={Draft},
	pdftex]{hyperref}

\usepackage{comment}

\usepackage{mathtools}

\usepackage{hyperref}
\usepackage[nameinlink]{cleveref}

\usepackage[all]{xy} 
\usepackage{tikz, times}
\usepackage{tkz-graph}
\usetikzlibrary{mindmap,backgrounds}
\tikzstyle{vertex}=[circle, draw, inner sep=0pt, minimum size=6pt]

\DeclareMathOperator{\GL}{GL}

\newtheorem{thm}{Theorem} 
\newtheorem{thmintro}{Theorem}

\newtheorem{prop}{Proposition} 
\newtheorem{cor}[prop]{Corollary} 
\newtheorem{lem}[prop]{Lemma} 

\theoremstyle{definition}

\newtheorem{ex}[prop]{Example}

\newcommand{\TT}{\mathbb{T}}

\newcommand{\RR}{\mathbb{R}}
\newcommand{\ZZ}{\mathbb{Z}}
\newcommand{\NN}{\mathbb{N}}

\newcommand{\HHH}{\mathcal{H}}
\newcommand{\bpm}{\begin{pmatrix}}
\newcommand{\epm}{\end{pmatrix}}

\newcommand{\sigmahat}{\hat{\sigma}}

\DeclareMathOperator{\SL}{SL}
\DeclareMathOperator{\pr}{pr}
\DeclareMathOperator{\sh}{sh}

\def\aff{\mathrm{Aff}}

\newcommand{\OriSquare}[5]
{
\begin{xy}
(0,0)="Pos";
"Pos"+(1,0) **@{-}; ?*h!U!/^2pt/{\text{\scriptsize #5}}, 
"Pos"+(1,1) **@{-}; ?*h!L!/^1pt/{\text{\scriptsize #2}},
"Pos"+(0,1) **@{-}; ?*h!D!/^1pt/{\text{\scriptsize #3}},
"Pos" **@{-}; ?*h!R!/^1pt/{\text{\scriptsize #4}},
"Pos"+(0.5,0.5) *h{\text{\scriptsize #1}};
\end{xy}
}

\title{Totally non congruence Veech groups}

\author{Jan-Christoph Schlage-Puchta, Gabriela Weitze-Schmithüsen\footnote{This work was supported by Project I.8 of SFB-TRR 195 'Symbolic Tools in Mathematics and their Application'}}

\begin{document}

\maketitle

\begin{abstract}
Veech groups are discrete subgroups of $\SL(2,\RR)$ which 
play an important role in the theory of translation surfaces.
For a special class of translation surfaces called origamis 
or square-tiled surfaces
their Veech groups are subgroups of finite index of $\SL(2,\ZZ)$.
We show that each stratum of the space of translation surfaces
contains infinitely many origamis whose Veech group is a totally 
non congruence group, i.e. it surjects to $\SL(2,\ZZ/n\ZZ)$
for any $n$.

\end{abstract}

\section{Introduction}
Within the last thirty years the study  of {\em translation surfaces} has 
become an active field in mathematics. Their moduli spaces come equipped 
with a natural action of $\SL(2,\RR)$.
It is  one of the  principal goals in this domain  to understand the orbits of this action.
This study culminated in the famous breakthrough result of Eskin, Mirzakhani and Mohammadi, namely  
the so-called {\em magic wand theorem}
 (cf. \cite{eskin+mirzakhani+mohammadi2015,eskin+mirzakhani}).
The {\em Veech group} $\Gamma(X,\mu)$ associated to a translation surface $(X,\mu)$
plays a crucial role in this topic. $\Gamma(X,\mu)$ is the stabiliser of $(X,\mu)$ under the action of 
$\SL(2,\RR)$. It turns out to be a  discrete subgroup of $\SL(2,\RR)$ and it 
carries a lot of information about the dynamical flow on the translation surface and about the 
Teichm\"uller flow defined by $(X,\mu)$. 
{\em Origamis} or {\em square-tiled surfaces}
are a particularly important class of translation surfaces. These surfaces
are tessellated by finitely many Euclidean unit squares. 
Their Veech groups
are especially easy to handle. They are subgroups of finite index of $\SL(2,\ZZ)$ 
and can be calculated explicitly from the combinatorial data which define the origami.
Furthermore, the set of origamis is dense in the moduli space of translation surfaces.
The action of $\SL(2,\RR)$ on the set of translation 
surfaces restricts to an action of $\SL(2,\ZZ)$ on origamis.\\
It is still open whether all subgroups of $\SL(2,\ZZ)$ of finite index occur as Veech groups
of origamis. A major result in this direction was achieved in \cite{ellenberg+mcreynolds2012} where it is proved
that all subgroups of finite index (satisfying a slight condition) of the principal congruence group 
$\Gamma(2)$ occur as Veech groups, where $\Gamma(2)$ is the group of matrices which 
are congruent to the identity matrix modulo 2. 
As a result in some sense in the opposite direction it is shown in \cite{Sc1} that all congruence groups
(cf. below) of prime level except five exceptions occur as Veech groups.\\
It is particularly interesting to study Veech groups of origamis that lie in the same  fixed stratum,
i.e. we fix the genus and the cone angles of the singularities (see below). 
\cite{hubert+lelievre2006, McM1} succeeded to give a complete classification of the $\SL(2,\ZZ)$-orbits 
of origamis in the stratum $\HHH_2(2)$ of translation surfaces of genus $2$ with one singularity of
angle $6\pi$. In this case, the set of origamis with $d$ squares decomposes depending on $d$
in one or two orbits.
There are only a few further classification results for certain subloci of strata 
(cf. \cite{lanneau+nguyen2014, lanneau+nguyenx2014,lanneau+nguyen2018}).
For general strata the classification problem is open.
However, there exists a conjecture for a precise description of the orbits in each stratum by Delecroix and Leli{\`e}vre
based on computer experiments.\\
A {\em congruence subgroup} $\Gamma$ of $\SL(2,\ZZ)$ is a subgroup which is fully determined by
its image in $\SL(2,\ZZ/n\ZZ)$ for some $n \in \NN$, i.e. it is the preimage of its image in
$\SL(2,\ZZ/n\ZZ)$ under the canonical projection $\SL(2,\ZZ) \to \SL(2,\ZZ/n\ZZ)$.
It turns out that such groups are rare among all finite index subgroups of $\SL(2,\ZZ)$.
Turning to Veech groups of origamis: there are several families of origamis
whose Veech groups could be explicitly determined as congruence groups in \cite{ExamplesVeechGroups,comb, Herrlich2006}.
In \cite{ExamplesNoncongruenceGroups} first examples of Veech groups that are non congruence groups
were detected. Hubert and Leli\`evre proved in \cite{HL2} that for all but  one 
origamis of genus 2 with one singularity their Veech group is a non congruence group. \\
For an arbitrary subgroup $\Gamma$ of $\SL(2,\ZZ)$ of finite index
we may measure how much information we loose, if
we consider all its images in the finite quotient groups $\SL(2,\ZZ/n\ZZ)$. In particular, all information
is lost if for all $n$ the image is the full group $\SL(2,\ZZ/n\ZZ)$. 
In this case, we call $\Gamma$ a {\em totally non congruence group}.
In \cite{deficiency} a criterion is given which detects totally noncongruence groups
(cf. \cite[Theorem~2]{deficiency}). It was further shown that in the stratum $\HHH_2(2)$
all Veech groups of origamis are totally non congruence groups or almost totally non congruence groups
(cf. \cite[Theorem 3]{deficiency}). Finally, it was shown that for each stratum $\HHH_{k+1}(2k)$ 
of translation surfaces with only one singularity of cone angle $(k+1)2\pi$ there are infinitely many
origamis whose Veech group is a totally non congruence group (cf. \cite[Theorem 4]{deficiency}).\\
In this article we generalise this statement to {\textbf all} strata. For this we first improve
the criterion for totally non congruence groups from \cite[Theorem 2]{deficiency} and  get
the following very handy conditions which assure that  a group $\Gamma$ is a totally non congruence group.

\begin{thmintro}
\setcounterref{thmintro}{new-criterion}
  Let $\Gamma$ be a finite index subgroup of $\SL(2,\ZZ)$. Denote
  $e_1 = {1 \choose 0}$. Suppose that for
  each prime $p$ there exist matrices $A_1, A_2 \in \SL(2,\ZZ)$ 
  with the following properties:
  \begin{itemize}
  \item[A)]  
    $\forall\; j \in \NN: A_1e_1 \neq j\cdot A_2e_1$ modulo $p$.
  \item[B)] 
    There exist $m_1,m_2 \in \NN$ with
    \[A_1T^{m_1}A_1^{-1} \mbox{ and } A_2T^{m_2}A_2^{-1} \mbox{ are contained in } \Gamma,\]
    such that 
    p divides neither $m_1$ nor $m_2$.
  \end{itemize}
  Then $\Gamma$ is a totally non congruence group.
\end{thmintro}

We then describe a method to construct one-cylinder origamis in each stratum
for which we have a good control over the cylinder decompositions in horizontal and
vertical direction and in the diagonal direction given by the vector $v = {1 \choose 1}$.
Choosing special elements of this family we finally prove the following theorem.
\begin{thmintro}
\setcounterref{thmintro}{thm:in_all_strata}
Every stratum contains an infinite family of origamis
whose Veech groups are totally non-congruence groups.
\end{thmintro}

\textbf{ Acknowledgement:} We would like to thank Martin M\"oller for helpful comments.

\section{Preliminaries}

In this section we give a concise introduction to translation surfaces,
origamis and Veech groups suited to the purpose of this article. You can find 
more elaborate introductions to this topic for example  in 
\cite{HS1, earle+gardiner1997, masur+tabachnikov2002,Sc1,Vorobets1997}. For the proofs 
of the facts that we state here we refer to these references.\\

{\textbf Translation surfaces, origamis and strata}
A {\em (finite) translation surface} is a surfaces $X$ with an atlas $\mu$
to $\RR^2$
such that all transition maps of the atlas $\mu$ are translations. 
The translation surface inherits a natural metric from the Euclidean 
metric in $\RR^2$. Furthermore we have a well-defined notion of directions,
since they are invariant under translations. Thus we may speak
for example of {\em horizontal} and {\em vertical geodesics}, or more general
of geodesics in direction $v \in \RR^2$. Moreover, using local charts we can
assign to each geodesic segment a vector in the plane $\RR^2$ which is
its {\em development vector}.
Let $\overline{X}$ be the metric
completion of $X$. The points in $\overline{X} \backslash X$ are called
{\em the singularities} of $X$. In this article, we consider the classical
situation of finite {\em translation surfaces}, i.e. translation surfaces
$(X,\mu)$ such that the metric completion is compact, the set of singularities is
discrete  and all singularities are 
cone points of finite cone angle $k2\pi$ ($k\in \NN$). A geodesic segment
between two (possible equal) singularities which does not contain any further singularity
is called a {\em saddle connection}.  
Further important geometric data of the translation surface $(X,\mu)$
are its set of closed geodesics and its set of maximal cylinders in a given direction
$v \in \RR^2$.  Here a maximal cylinder is a maximal connected set of homotopic simple
closed geodesics.
For genus $g \geq 2$, every closed geodesic lies in a unique maximal cylinder in the direction 
$v$ of the geodesic which is bounded by saddle connections, since we may move the 
geodesic transversely to $v$ until we hit singularities.\\
Finite translation surfaces 
are naturally distinguished into strata by their type of singularities. More precisely, 
a finite translation surface $(X,\mu)$ is said to be {\em of 
type $(\alpha_1,\ldots, \alpha_n)$}, if $\overline{X}$ has $n$ singularities
of cone angle $(\alpha_1+1)\cdot2\pi$, \ldots, $(\alpha_n+1)\cdot2\pi$.
The usage of $\alpha_i$ instead of $\alpha_i+1$ relates to the fact that 
a finite translation surface can equivalently 
be defined as a closed Riemann surface $X$ together with a holomorphic 
differential $\omega$. The charts of the atlas are then obtained by integrating 
with respect to $\omega$, the singularities are the zeroes of $\omega$
and $\alpha_i$ is the order of the zero. We then define the {\em stratum} 
$\HHH_g(\alpha_1,\ldots, \alpha_n)$
as the set of all equivalence classes of translation surfaces of type $(\alpha_1,\ldots, \alpha_n)$
of genus $g$.
Two translation surfaces $(X_1,\mu_1)$ and $(X_2,\mu_2)$ are {\em equivalent}, if there exists
a  translation $f: X_1 \to X_2$, i.e. a  homeomorphism which is a 
translation on each chart. We will usually write $(X,\mu) \in \HHH_g(\alpha_1,\ldots, \alpha_r)$
for the equivalence class defined by $(X,\mu)$.
The set $\HHH_g(\alpha_1,\ldots, \alpha_n)$ is endowed with a topology
itself, more precisely there is a natural way to define local coordinates as a manifold on a covering of 
it  (cf.~\cite[Section 6.3]{yoccoz2010}). Furthermore, $\HHH_g(\alpha_1,\ldots, \alpha_n)$ is endowed with a natural action
of $\SL(2,\RR)$ as follows. For a translation surface $(X,\mu)$ and a matrix $A \in \SL(2,\RR)$,
we define $A\cdot(X,\mu) = (X,\mu_A)$ to be the translation surface obtained from 
$(X,\mu)$ by composing each chart of $\mu$ with the linear map $z \mapsto A\cdot z$.
It is one of the main objectives in the field to understand the orbits 
of this action.\\
There is yet an other way how to define finite translation surfaces:
Take finitely many polygons in the plane such that their edges come in pairs of edges
of same length and same direction. Glue for each pair its two edges by a translation. 
In this way we obtain a closed surface $\overline{X}$. 
The points which come from the vertices of the polygons may be cone points. Removing them
defines a translation surface $X$. 
If all the polygons which form the translation surface are copies of the Euclidean 
unit square, the translation surface is called an {\em origami} or a {\em square-tiled surface}
(cf. Figure~\ref{figure:WMS}). 
\\[3mm]

\begin{center}
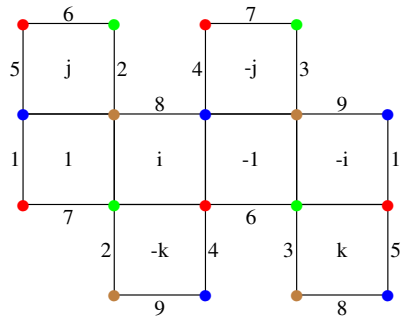
\begin{figure}[h]
\centering
\hspace*{1mm} 
\begin{xy}
<1.2cm,0cm>:
(0,1)*{\OriSquare{1}{}{}{1}{7}};
(1,1)*{\OriSquare{i}{}{8}{}{}};
(2,1)*{\OriSquare{-1}{}{}{}{6}};
(3,1)*{\OriSquare{-i}{1}{9}{}{}};
(1,0)*{\OriSquare{-k}{4}{}{2}{9}};
(2,2)*{\OriSquare{-j}{3}{7}{4}{}};
(3,0)*{\OriSquare{k}{5}{}{3}{8}};
(0,2)*{\OriSquare{j}{2}{6}{5}{}};
(-0.5,0.52)*{{\color{red}\bullet}};
(3.5,0.52)*{{\color{red}\bullet}};
(-0.5,2.52)*{{\color{red}\bullet}};
(1.5,.52)*{{\color{red}\bullet}};
(1.5,2.52)*{{\color{red}\bullet}};
(0.5,0.52)*{{\color{green}\bullet}};
(2.5,2.52)*{{\color{green}\bullet}};
(2.5,0.52)*{{\color{green}\bullet}};
(.5,2.52)*{{\color{green}\bullet}};
(-0.5,1.52)*{{\color{blue}\bullet}};
(3.5,1.52)*{{\color{blue}\bullet}};
(1.5,-.48)*{{\color{blue}\bullet}};
(1.5,1.52)*{{\color{blue}\bullet}};
(3.5,-.48)*{{\color{blue}\bullet}};
(0.5,1.52)*{{\color{brown}\bullet}};
(2.5,-.48)*{{\color{brown}\bullet}};
(2.5,1.52)*{{\color{brown}\bullet}};
(.5,-.48)*{{\color{brown}\bullet}};
\end{xy}
\caption{Gluing edges with same labels defines an origami of genus 3. This origami steams from \cite{WMS}.}
\label{figure:WMS}
\end{figure}
\end{center}

{\textbf Veech groups and the action of \boldmath{$\SL(2,\mathbb{R})$}} Let $(X,\mu)$ be a finite translation surface
of genus $g$ in some stratum $\HHH_g(\alpha_1, \ldots, \alpha_n)$.
The {\em Veech group} $\Gamma(X,\mu)$ is the stabiliser of $(X,\mu)$ for the action of $\SL(2,\RR)$
on $\HHH_g(\alpha_1, \ldots, \alpha_n)$ .  It can equivalently be
defined in the following way. Consider the group Aff$(X,\mu)$
of all {\em affine homeomorphisms} of $X$, i.e. homeomorphisms which are with respect to  charts
of the form $z \mapsto A\cdot z + b$ with $A \in \SL(2,\RR)$ and $b \in \RR^2$. It turns out
that since all transition maps are translations the matrix $A$ is independent of the chosen charts.
We obtain a group homomorphism $D: \aff(X,\mu) \to \SL(2,\RR)$ which maps the affine homeomorphism $f$
to the matrix $A$, i.e. to its derivative. 
The Veech group is the image of $D$, hence it consists of
all matrices $A$ which occur as derivative of some affine homeomorphism of the surface. It was already
shown by Veech himself that $\Gamma(X,\mu)$ 
is a discrete subgroup of $\SL(2,\RR)$ (cf. \cite[Proposition 2.7]{Veech3} 
or \cite[Proposition 3.3]{Vorobets1997} for a very nice presentation). Furthermore, two translation surfaces
in the same $\SL(2,\RR)$-orbit have conjugated Veech groups.\\
Let us consider the example of the torus $\RR^2/\ZZ^2$ endowed with the translation structure 
of its universal covering $\RR^2$. Observe that the affine homeomorphisms
lift to affine homeomorphisms of $\RR^2$ which preserve the lattice $\ZZ^2$ up to a translation. 
And all such maps descend to the torus.
Therefore the Veech group is in this case $\SL(2,\ZZ)$.\\[3mm]

{\textbf Special properties of origamis.}
We will use three equivalent ways to describe origamis, as explained in the following.
The equivalences are described in more details in \cite[Section 1]{ExamplesNoncongruenceGroups}.
Recall that we obtain an origami by gluing copies of the Euclidean 
unit square along their edges which leads to a closed surface $\overline{X}$
tiled by squares. Hence, an origami made from $d$ unit squares is fully 
determined by a pair of permutations $(\sigma_a,\sigma_b)$
as follows. We label the squares with $\{1,\ldots, d\}$, then $\sigma_a(i)$ and $\sigma_b(i)$
denote the right and the upper neighbour of the square labelled by $i \in \{1,\ldots, d\}$.
The fact that the surface is connected is equivalent to the fact that
the subgroup of $S_d$ generated by the two permutations $\sigma_a$ and $\sigma_b$
acts transitively on the set $\{1,\ldots, d\}$.
If we choose an other labelling of the squares this leads to a simultaneous conjugation of the pair
of permutations $(\sigma_a,\sigma_b)$. Altogether, we obtain an equivalence between the set of origamis
up to translation equivalence and the set of pairs $(\sigma_a,\sigma_b)$ in $S_d^2$
up to simultaneous conjugation. There is yet an other equivalent description of origamis
which we will use. Observe that the surface $\overline{X}$ comes with a
covering $p$ to the square-torus $\TT$ obtained by gluing parallel edges of the unit square. Namely, 
we map each square on $\overline{X}$ to the one square forming $\TT$ and this map 
is well-defined with respect to the gluings. The map $p$ is an unramified covering for
all points which are not vertices. Hence if $\infty \in \TT$ is the point
obtained from the vertex of the unit square, then $p: \overline{X} \to \TT$
is ramified at most over $\infty$.\\[3mm]
For an origami $(X,\mu)$ the Veech group is always a finite index subgroup of $\SL(2,\ZZ)$. 
Here we should point to a subtlety in the definition of origami.
Recall that we obtain the origami by gluing copies of the Euclidean unit square along their edges.
More precisely, this gives us the metric completion $\overline{X}$ of the translation surface. 
The singularities of the translation surface stem from the vertices of the squares.
However  not every vertex has to be a singularity. 
Now there are two different natural ways how two define the translation surfaces $X$. We might either 
remove only the singularities of $\overline{X}$, or we might remove all points which come from a vertex.
In the second case the Veech group is indeed a subgroup of $\SL(2,\ZZ)$ of finite index, 
in the first case it is commensurable to $\SL(2,\ZZ)$. However it turns out that
for {\em reduced origamis} one obtains equal  Veech groups for the translation surface with only
singularities removed and for the surface with all vertex points removed (cf. \cite[Remark 2.9]{KappesPhDthesis}). 
Following \cite[Section 1.2]{matheus+moeller+yoccoz2015},
we call an origami {\em reduced}, if the set of development vectors of all saddle connections generate $\ZZ^2$.
This is a very mild restriction, since any origami 
$O$ is affine equivalent to a reduced origami $O'$, i.e. there is some matrix $A \in \GL(2,\RR)$
such that $O' \sim A\cdot O$ and thus their Veech groups are conjugated in $\GL(2,\RR)$. 
Here  the action of $\GL(2,\RR)$ on translation surfaces is defined just in the 
same way as the action of $\SL(2,\RR)$. In this article we will restrict to reduced origamis
and thus all Veech groups are subgroups of $\SL(2,\ZZ)$ of finite index.\\[3mm]
Suppose that an origami is now given by the pair of permutations $(\sigma_a,\sigma_b)$.
We obtain the stratum in which the associated translation surface lives 
in the following way: Let $p:\overline{X} \to \TT$ be the corresponding ramified cover of the torus.
Let us choose a loop around the vertex of $\TT$, namely $xyx^{-1}y^{-1}$, where
$x$ and $y$ are the closed curves on $\TT$ shown in \Cref{torus}. 
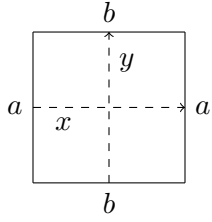
\begin{figure}
\begin{center}
\begin{tikzpicture}
\draw (0,0) -- (0,2) node[midway,left]{$a$};
\draw (0,2)  -- (2,2) node[midway,above]{$b$};
\draw (2,2)  -- (2,0) node[midway, right]{$a$};
\draw (0,0) -- (2,0) node[midway, below]{$b$};
\draw[dashed, ->] (0,1) -- (2,1) node[pos=.2, below]{$x$};
\draw[dashed, ->] (1,0) -- (1,2) node[pos=.8, right]{$y$};
\end{tikzpicture}
\caption{A torus $T$ with the standard system of generators of the fundamental group}
\label{torus}  
\end{center}
\end{figure}
The connected components of
the preimage of this curve are loops around the singularities. Hence the number of the
connected components are the number of singularities. Furthermore if the multiplicity of 
a component is $k$, then the corresponding singularity is of angle $2k\pi$. 
Hence the commutator $[\sigma_b^{-1},\sigma_a^{-1}]$ determines the type of the singularities that we obtain.
More precisely, each cycle of length $k$ in the commutator corresponds to a singularity of cone
angle $k\cdot 2\pi$. \\
In the proof our result the following two facts are crucial which are described in more detail e.g.
in \cite[2.2,2.3]{deficiency}:
\begin{enumerate}
\item
  The action  of $\SL(2,\RR)$ on translation surfaces restricts to an action of $\SL(2,\ZZ)$
  on origamis. The action can be explicitly given as described in the following. The two generators 
  \[S = \begin{pmatrix} 0&-1\\1&0 \end{pmatrix} \mbox{ and } T = \begin{pmatrix} 1&1\\0&1 \end{pmatrix}\]  
  act  on an origami given as pair of permutations $(\sigma_a,\sigma_b)$ in the following way:
  \[S : (\sigma_a,\sigma_b) \mapsto (\sigma_b^{-1}, \sigma_a^{-1}) \mbox{ and } 
  T: (\sigma_a,\sigma_b) \mapsto (\sigma_a,\sigma_b\sigma_a^{-1})\]
\item
  Suppose that the translation surface $(X,\mu)$ defined by a primitive origami $O$ 
  decomposes in the horizontal direction into $k$ cylinders of height 1 and length $m_1, \ldots, m_k$
  and let $m$ be a multiple of $m_1$, \ldots, $m_k$. Then $T^m$ is in the Veech group $\Gamma(O)$.
  Similarly, if $(X,\mu)$ decomposes in the vertical direction into $l$ cylinders of length $m'_1$, \ldots, $m'_l$ and $m'$
  is a multiple of  $m'_1$, \ldots, $m'_l$, then $\Gamma(X,\mu)$ contains $T'^{m'}$. Here 
  \[T' = \begin{pmatrix}1&0\\1&1\end{pmatrix}.\]
    If $O$ is given by the pair of permutations
  $(\sigma_a,\sigma_b)$, then the numbers $m_1$, \ldots, $m_k$
  are precisely the cycle lengths of $\sigma_a$ and $m'_1$, \ldots, $m'_l$ are the cycle lengths of $\sigma_b$.
\end{enumerate}

\section{A criterion for being a totally non congruence group}

We denote 
\begin{equation}\label{defT}
  T = \begin{pmatrix} 1&1\\0&1 \end{pmatrix}, T' = \begin{pmatrix} 1&0\\1&1 \end{pmatrix} \mbox{ and } 
  S = \begin{pmatrix} 0&-1\\1&0 \end{pmatrix} .
\end{equation}
Furthermore $p_n: \SL(2,\ZZ) \to \SL(2,\ZZ/ n\ZZ)$ is the canonical projection. We denote
the images of the matrices $T$ and $T'$ in $\SL(2,\ZZ/ n\ZZ)$ also by $T$ and $T'$. Finally, we denote by
$I$ the $2\times 2$-identity matrix over the respective ring. \\

We start with a small but very useful calculation.

\begin{lem}\label{conjugation}
  Let $A,B \in \GL(2,\ZZ/n\ZZ)$ with $A \cdot {1 \choose 0} = B  \cdot {1 \choose 0}$.
  Then we have that
  \[ ATA^{-1} = BT^{\det(B)/\det(A)}B^{-1}.
  \]
\end{lem}

Observe for the statement in \Cref{conjugation} that $T^a$ with $a \in \ZZ/n\ZZ$ gives a well-defined
matrix in $\GL(2,\ZZ/n\ZZ)$ and we have 
for any $A \in \GL(2,\ZZ/n\ZZ)$ that $AT^aA^{-1} = (ATA^{-1})^a$.

\begin{proof}
  Suppose first that $A{1 \choose 0} = B{1 \choose 0} = {1 \choose 0}$.
  Hence  we can write
  \[
  A= \begin{pmatrix} 1&x\\0&\det(A)\end{pmatrix} \mbox{ and }
  B = \begin{pmatrix} 1&y\\0&\det(B)\end{pmatrix}
  \] 
  with $x, y \in \ZZ/n\ZZ$.
    A short calculation gives:
    \[ 
    ATA^{-1} = \begin{pmatrix}1 &\det(A)^{-1}\\0&1\end{pmatrix} \mbox{ and }
    BTB^{-1} = \begin{pmatrix}1 &\det(B)^{-1}\\0&1\end{pmatrix}\]
  Thus the claim holds in this case. 
  In the general situation we consider the two matrices $A^{-1}B$ and $I$ satisfying
  $A^{-1}B {1 \choose 0} = I{1 \choose 0} = {1 \choose 0}$ and obtain from the  
  preceding consideration
  \[T = (A^{-1}B)T^{\det(A^{-1}B)}(B^{-1}A) = A^{-1}BT^{\det(B)/\det(A)}B^{-1}A,\]
  which implies the claim.
\end{proof}

We now deduct from \Cref{conjugation} a criterion whether 
two conjugates of $T$ generate the full group $\SL(2,\ZZ/p^r\ZZ)$.

\begin{lem}\label{lemma:generate}
Let  $p$ be prime and $r \in \NN$.
Let $\Gamma$ be a subgroup of  $\SL(2,\ZZ/p^r\ZZ)$.
Suppose that $\Gamma$ contains $A_1TA_1^{-1}$ and $A_2TA_2^{-1}$
with $A_1, A_2 \in \SL(2,\ZZ/p^r\ZZ)$ such that
\[\forall m \in \NN:\; mA_1e_1 \neq  A_2e_1 \mod p, \mbox{ where } e_1 = {1 \choose 0} \in (\ZZ/p^r\ZZ)^2\] 
Then $\Gamma = \SL(2,\ZZ/p^r\ZZ)$.
\end{lem}

\begin{proof}
By conjugation we may assume that $A_1 = I$ is the identity matrix.
Consider the vector ${a \choose c} = A_2\cdot e_1$. By assumption $c$
is not divisible by $p$, hence $c$ is in the multiplicative group $(\ZZ/p^r\ZZ)^{\times}$. 
Consider the following matrix $B \in \GL(2,\ZZ/p^r\ZZ)$ and its inverse $B^{-1}$:
\[B = \begin{pmatrix} 1&a\\0&c \end{pmatrix} \mbox{ and } B ^{-1} = c^{-1} \cdot \begin{pmatrix}c&-a\\0&1 \end{pmatrix}.\]
It follows directly from the definition of $B$ that
\[B^{-1}e_1 = e_1 \mbox{ and } B^{-1}A_2e_1 = e_2 = Se_1 \mbox{, where } e_2 = {0 \choose 1}\]
Hence we obtain from \Cref{conjugation}:
\begin{equation}\label{eqn:contains}
  (B^{-1}TB)^{\det(B^{-1})} = T \mbox{ and } (B^{-1}A_2TA_2^{-1}B)^{\det(B^{-1})} = STS^{-1} = T'^{-1}
\end{equation} 
It follows that \[\SL(2,\ZZ/p^r\ZZ) \; = \;\; <T,T'> \;\; \subseteq \;  B^{-1}\Gamma B \; \subseteq \; \SL(2,\ZZ/p^r\ZZ).\] 
Hence we have $B^{-1}\Gamma B  = \SL(2,\ZZ/p^r\ZZ)$ and thus $\Gamma = \SL(2,\ZZ/p^r\ZZ)$.
Here it is crucial that $\SL(2,\ZZ/p^r\ZZ)$ is a normal subgroup in $\GL(2,\ZZ/p^r\ZZ)$.
\end{proof}

\Cref{lemma:generate} is the main ingredient that we need
to prove \Cref{new-criterion}, which provides us with a criterion
whether a group is a totally non congruence group.

\begin{thm}\label{new-criterion}
  Let $\Gamma$ be a finite index subgroup of $\SL(2,\ZZ)$. Denote
  $e_1 = {1 \choose 0}$. Suppose that for
  each prime $p$ there exist matrices $A_1, A_2 \in \SL(2,\ZZ)$ 
  with the following properties:
  \begin{itemize}
  \item[A)]  
    $\forall\; j \in \NN: A_1e_1 \neq j\cdot A_2e_1$ modulo $p$.
  \item[B)] 
    There exist $m_1,m_2 \in \NN$ with
    \[A_1T^{m_1}A_1^{-1} \mbox{ and } A_2T^{m_2}A_2^{-1} \mbox{ are contained in } \Gamma,\]
    such that 
    p divides neither $m_1$ nor $m_2$.
  \end{itemize}
  Then $\Gamma$ is a totally non congruence group.
\end{thm}

\begin{proof}
  We have to show that $\pr_n(\Gamma) = \SL(2,\ZZ/n\ZZ)$ for all $n\in\NN$.\\
  Let $n = {p_1}^{\hspace*{-.7mm}r_1}\cdot\ldots\cdot{\hspace*{-.7mm}p_k}^{r_k}$ be the prime factorisation of $n$.
  We thus have by the Chinese remainder theorem:
  \[ \SL(2,\ZZ/n\ZZ) = \SL(2,\ZZ/{p_1}^{\hspace*{-.7mm}r_1}\ZZ) \times  \ldots \times \SL(2,\ZZ/{p_k}^{\hspace*{-.7mm}r_k}\ZZ)\]
  We show: 
  \begin{eqnarray}\label{eqn:surjects}
    \forall i \in \{1,\ldots, k\}:
           \pr_n(\Gamma) \supseteq \{I\} \times \ldots \times \{I\} \times  \SL(2,\ZZ/{p_i}^{\hspace*{-.7mm}r_i}\ZZ) 
            \times \{I\} \times \ldots \times \{I\}.
  \end{eqnarray}
  For $p = p_i$ we decompose $n = {p}^{r} \cdot n_2$ with $\gcd(p,n_2) = 1$. 
  Choose $m_1$, $m_2$ such that they satisfy assumptions A) and B) with respect to $p$.
  In particular, $m_1$ and $m_2$ are coprime to $p$.
  By B{\'e}zout's identity we find $a,b \in \ZZ$ with 
  $1 = a\cdot p^{r}+ b\cdot m_1 m_2 n_2$.\\
  We then have for $K = bm_1m_2n_2$ that 
  \[\Gamma \ni A_1T^{K}A_1^{-1}  = A_1\begin{pmatrix}1&bm_1m_2n_2\\0&1\end{pmatrix}A_1^{-1}.\]
  Furthermore, we have:
  \[
    \begin{array}{lcl}
     A_1T^KA_1^{-1} &\equiv& A_1\cdot \begin{pmatrix} 1&1\\0&1 \end{pmatrix}A_1^{-1} = A_1TA_1^{-1} \mod p^{r} \mbox{ and }\\
     A_1T^KA_1^{-1} &\equiv& I \mod n_2.
    \end{array}
    \]    
    Hence the group $\pr_n(\Gamma)$ contains
   \[\pr_n(A_1T^KA_1^{-1}) = (A_1TA_1^{-1},I) \in 
     \SL(2,\ZZ/{p}^{r}\ZZ) \times \SL(2,\ZZ/n_2\ZZ) = \SL(2,\ZZ/n\ZZ) .\]
   Similarly, we obtain that 
   \[ \pr_n(\Gamma) \ni \pr_n(A_2T^KA_2^{-1}) = (A_2TA_2^{-1},I) \in  
      \SL(2,\ZZ/{p}^{r}\ZZ) \times \SL(2,\ZZ/n_2\ZZ) =  \SL(2,\ZZ/n\ZZ) .\]
   It follows from \Cref{lemma:generate} that
   \[\pr_n(\Gamma) \supseteq  \SL(2,\ZZ/{p}^{r}\ZZ) \times \{I\}.\]
   This implies the claim.
\end{proof}

\Cref{new-criterion} is a generalisation of  \cite[Theorem~2]{deficiency} which we restate
adapted to our context in \Cref{cor:deficiency}.

\begin{cor}{\cite[Theorem~2]{deficiency}}  \label{cor:deficiency}\newline
Let $\Gamma$ be a finite index subgroup of $\SL(2,\ZZ)$.
Suppose there exist matrices $C_1, C_2 \in \SL(2,\ZZ)$ and $m_1, m_1', m_2, m_2' \in \NN$
with
\[\Gamma \; \ni \; C_1T^{m_1}C_1^{-1}, C_1T'^{\,m_1'}C_1^{-1} \mbox{ and } \Gamma \; \ni \; C_2T^{m_2}C_2^{-1}, C_2T'^{\,m_2'}C_2^{-1}, \]
such that $\gcd(m_1m_1',m_2m_2') = 1$.
Then $\Gamma$ is a totally non-congruence group.
\end{cor}

\begin{proof}
We show that the assumptions of \Cref{new-criterion} are fulfilled.
Let $p$  be prime. If $p$ does not divide $m_1m_1'$, then we may choose $A_1 = C_1$, $A_2 = C_1S$.
Denote $e_1 = {1 \choose 0}$ and $e_2 = {0 \choose 1}$. Since $e_1 \neq j\cdot e_2 \mod p$ 
for all $j \in \NN$, we have
that $A_1e_1 \neq j\cdot A_2e_1 = j\cdot A_1e_2 \mod p$. Thus in this case the assumptions are satisfied.
If $p$ divides $m_1m_1'$, then it does not divide $m_2m_2'$ and we can use the same   
arguments with $C_2$ instead of $C_1$.
\end{proof}

\section{Nice one-cylinder origamis}

In this section we give explicit examples for one-cylinder origamis
in each stratum. The following examples will provide building blocks for them.

\begin{ex}\label{example:building_blocks}
In the following we construct special one-cylinder origamis 
in $\HHH(\alpha)$ with $\alpha$ even
and in $\HHH(\alpha_1,\alpha_2)$ with $\alpha_1, \alpha_2$ odd.
\begin{itemize}
\item[i)] {\em A family of  origamis in $\HHH(\alpha)$}:\\[2mm]
  Let $\alpha = 2k$ be an even number. Define the origami $O(\alpha)$ 
  with $N = 3k+1 = \frac{3}{2}\alpha + 1$ squares 
  by the following permutations (cf.~Figure~\ref{origami1}):
  \[\begin{array}{lcl}
  \sigma_a(\alpha) &=& (1, \ldots, N),\\
  \sigma_b(\alpha) &=& (1,2,3)\circ (4,5,6) \circ \ldots \circ (3(k-1)+1, 3(k-1)+2,3(k-1)+3)
  \end{array}\]
  Observe that we obtain the commutator
  \[[\sigma_b^{-1},\sigma_a^{-1}] = (3, 6, 9, \ldots, 3k-1, 3k, 3k-1, 3k-4, 3k-7, \ldots, 8,5,2,N)\]
  In particular the commutator consists of one cycle of length $2k+1$. Hence the origami has one singularity
  with cone angle $(2k+1)\cdot 2\pi = (\alpha+1)\cdot 2\pi$ and thus lies in $\HHH(\alpha)$.\\[3mm]
  \begin{figure}
    \begin{center}
      \hspace*{1mm}
      \begin{xy}
        <.9cm,0cm>:
        (0,0)*{\OriSquare{$1$}{}{$a_1$}{$x$}{$c_1$}};
        (1,0)*{\OriSquare{$2$}{}{$b_1$}{}{$a_1$}};
        (2,0)*{\OriSquare{$3$}{}{$c_1$}{}{$b_1$}};
        (3,0)*{\OriSquare{$4$}{}{$a_2$}{}{$c_2$}};
        (4,0)*{\OriSquare{$5$}{}{$b_2$}{}{$a_2$}};
        (5,0)*{\OriSquare{$6$}{}{$c_2$}{}{$b_2$}};
        (6,0)*{\OriSquare{$\cdots$}{}{$\cdots$}{}{$\cdots$}};
        (7,0)*{\OriSquare{$\cdots$}{}{$\cdots$}{}{$\cdots$}};
        (8,0)*{\OriSquare{$\cdots$}{}{$\cdots$}{}{$\cdots$}};
        (9,0)*{\OriSquare{$3k-2$}{}{$a_k$}{}{$c_k$}};
        (10,0)*{\OriSquare{$3k-1$}{}{$b_k$}{}{$a_k$}};
        (11,0)*{\OriSquare{$3k$}{}{$c_k$}{}{$b_k$}};
        (12,0)*{\OriSquare{$3k+1$}{$x$}{$d$}{}{$d$}};
        (-.5,.5)*{\bullet};
        (1.5,.5)*{\bullet};
        (2.5,.5)*{\bullet};
        (4.5,.5)*{\bullet};
        (5.5,.5)*{\bullet};
        (7.5,.5)*{\bullet};
        (8.5,.5)*{\bullet};
        (10.5,.5)*{\bullet};
        (11.5,.5)*{\bullet};
        (12.5,.5)*{\bullet};
        (-.5,-.5)*{\bullet};
        (.5,-.5)*{\bullet};
        (2.5,-.5)*{\bullet};
        (3.5,-.5)*{\bullet};
        (5.5,-.5)*{\bullet};
        (6.5,-.5)*{\bullet};
        (8.5,-.5)*{\bullet};
        (9.5,-.5)*{\bullet};
        (11.5,-.5)*{\bullet};
        (12.5,-.5)*{\bullet};
      \end{xy}
      \caption{The origami $O(\alpha)$ from \Cref{example:building_blocks} in $\HHH(\alpha)$}
      \label{origami1}
    \end{center}
  \end{figure}
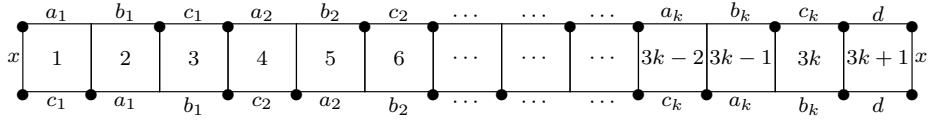
  We now define for arbitrary $l \geq 1$ the one-cylinder origami $O(\alpha;l)$  in $\HHH(\alpha)$
  as a deformation of $O(\alpha)$ in the following way (cf.~Figure~\ref{origami2}). $O(\alpha;l)$ has 
  $N' = N + l - 1 = \frac{3}{2}\alpha + l$ squares and is defined by the permutations
  \[
  \begin{array}{lcl}
    \sigma_{a}(\alpha;l) &=& (1,\ldots, N'),\\
    \sigma_{b}(\alpha;l) &=& \sigma_b(\alpha) = (1,2,3)\circ (4,5,6) \circ \ldots \circ (3(k-1)+1, 3(k-1)+2,3(k-1)+3)
  \end{array}
  \]
  Observe that  $O(\alpha;l)$ has again one singularity and lies in $\HHH(\alpha)$.\\

\begin{figure}
    \begin{center}
      \hspace*{1mm}
      \begin{xy}
        <.9cm,0cm>:
        (0,0)*{\OriSquare{$1$}{}{$a_1$}{$x$}{$c_1$}};
        (1,0)*{\OriSquare{$2$}{}{$b_1$}{}{$a_1$}};
        (2,0)*{\OriSquare{$3$}{}{$c_1$}{}{$b_1$}};
        (3,0)*{\OriSquare{$4$}{}{$a_2$}{}{$c_2$}};
        (4,0)*{\OriSquare{$5$}{}{$b_2$}{}{$a_2$}};
        (5,0)*{\OriSquare{$6$}{}{$c_2$}{}{$b_2$}};
        (6,0)*{\OriSquare{$\cdots$}{}{$\cdots$}{}{$\cdots$}};
        (7,0)*{\OriSquare{$\cdots$}{}{$\cdots$}{}{$\cdots$}};
        (8,0)*{\OriSquare{$\cdots$}{}{$\cdots$}{}{$\cdots$}};
        (9,0)*{\OriSquare{$3k-2$}{}{$a_k$}{}{$c_k$}};
        (10,0)*{\OriSquare{$3k-1$}{}{$b_k$}{}{$a_k$}};
        (11,0)*{\OriSquare{$3k$}{}{$c_k$}{}{$b_k$}};
        (12,0)*{\OriSquare{$3k+1$}{}{$d_1$}{}{$d_1$}};
        (13,0)*{\OriSquare{$\ldots$}{}{$\cdots$}{}{$\cdots$}};
        (14,0)*{\OriSquare{$3k+l$}{$x$}{$d_l$}{}{$d_l$}};
        (-.5,.5)*{\bullet};
        (1.5,.5)*{\bullet};
        (2.5,.5)*{\bullet};
        (4.5,.5)*{\bullet};
        (5.5,.5)*{\bullet};
        (7.5,.5)*{\bullet};
        (8.5,.5)*{\bullet};
        (10.5,.5)*{\bullet};
        (11.5,.5)*{\bullet};
        (14.5,.5)*{\bullet};
        (-.5,-.5)*{\bullet};
        (.5,-.5)*{\bullet};
        (2.5,-.5)*{\bullet};
        (3.5,-.5)*{\bullet};
        (5.5,-.5)*{\bullet};
        (6.5,-.5)*{\bullet};
        (8.5,-.5)*{\bullet};
        (9.5,-.5)*{\bullet};
        (11.5,-.5)*{\bullet};
        (14.5,-.5)*{\bullet};
      \end{xy}
      \caption{The origami $O(\alpha;l)$ from \Cref{example:building_blocks} in $\HHH(\alpha)$}
      \label{origami2}
    \end{center}
  \end{figure}
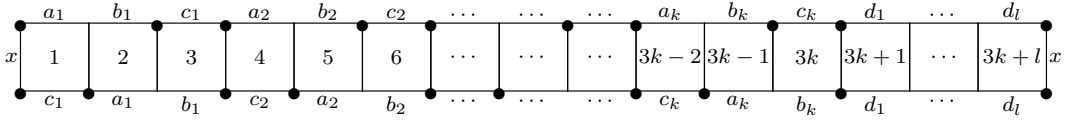

\item[ii)]  {\em A family of  origamis in $\HHH(\alpha_1,\alpha_2)$} (cf.~\Cref{origami3}):\\[2mm]
  Let $\alpha_1 = 2k_1+1$, $\alpha_2 = 2k_2+1$ be odd numbers. 
  Define the origami $O(\alpha_1,\alpha_2)$ 
  with $N = 3(k_1+k_2) + 6 = \frac{3}{2}(\alpha_1 + \alpha_2) + 3$ squares
  by the following permutations: 
  \[
  \begin{array}[t]{c}
    \sigma_a(\alpha_1,\alpha_2) = (1, \ldots, N),\quad
    \sigma_b(\alpha_1,\alpha_2) = \sigma_1 \circ \sigma_2 \circ \sigma_3,
  \end{array}
  \]
  \mbox{where}
  $\begin{array}[t]{lcl}
    \sigma_1 &=& (1,2,3) \circ (4,5,6) \circ \ldots \circ (3k_1-2,3k_1-1,3k_1),\\ 
    \sigma_2 &=& (3k_1+1,3k_1+5,3k_1+2,3k_1+3, 3k_1+4),\\
    \sigma_3 &=& (3k_1+6,3k_1+7,3k_1+8)\circ (3k_1+9,3k_1+10,3k_1+11)\circ\ldots\\ 
    & & \hspace*{5mm}\ldots \circ (3(k_1+k_2)+3,3(k_1+k_2)+4,3(k_1+k_2)+5)
  \end{array}$\\[3mm]
  In this case we obtain  for the commutator:
  \[[\sigma_b^{-1},\sigma_a^{-1}] = 
        \begin{array}[t]{l}
             (3,6,9, \ldots,3k_1,3k_1+3, 3k_1-1, 3k_1-4, 3k_1 - 7, \ldots, 5,2,N) \circ \\
              (3k_1+1, 3k_1+5, 3k_1+8, 3k_1+11, \ldots, N-1,\\
               \hspace*{6mm} N-2, N-5, N-8, \ldots, 3k_1 + 5 + 2) 
        \end{array} 
  \]
  In particular it consists of two cycles of length $2k_1+2 = \alpha_1+1$
  and $2k_2+2 = \alpha_2+1$. Hence $O(\alpha_1,\alpha_2)$ lies in $\HHH(\alpha_1,\alpha_2)$.
  Similarly as in i), we define for $l \geq 1$ the origami $O(\alpha_1,\alpha_2;l)$ in $\HHH(\alpha_1,\alpha_2)$ with
  $N' = 3(k_1+k_2) + 5 + l =  \frac{3}{2}(\alpha_1 + \alpha_2) + 2  + l$ squares by the two permutations
   (cf.~\Cref{origami3})
  \[
  \begin{array}{lcl}
    \sigma_{a}(\alpha_1,\alpha_2;l) &=& (1,\ldots, N'),\\
    \sigma_{b}(\alpha_1,\alpha_2;l) &=&  \sigma_{b}(\alpha_1,\alpha_2) 
  \end{array}
  \]

  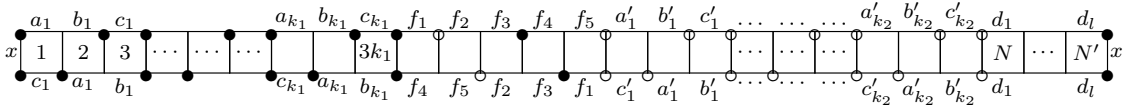
\begin{figure}

    \hspace*{10mm}
    \begin{center}
      \begin{xy}
        \begin{xy}
          <5.5mm,0mm>:
          (0,0)*{\OriSquare{$1$}{}{$a_1$}{$x$}{$c_1$}};
          (1,0)*{\OriSquare{$2$}{}{$b_1$}{}{$a_1$}};
          (2,0)*{\OriSquare{$3$}{}{$c_1$}{}{$b_1$}};
          (3,0)*{\OriSquare{$\cdots$}{}{}{}{}};
          (4,0)*{\OriSquare{$\cdots$}{}{}{}{}};
          (5,0)*{\OriSquare{$\cdots$}{}{}{}{}};
          (6,0)*{\OriSquare{}{}{$a_{k_1}$}{}{$c_{k_1}$}};
          (7,0)*{\OriSquare{}{}{$b_{k_1}$}{}{$a_{k_1}$}};
          (8,0)*{\OriSquare{$3k_1$}{}{$c_{k_1}$}{}{$b_{k_1}$}};
          (9,0)*{\OriSquare{}{}{$f_1$}{}{$f_4$}};
          (10,0)*{\OriSquare{}{}{$f_2$}{}{$f_5$}};
          (11,0)*{\OriSquare{}{}{$f_3$}{}{$f_2$}};
          (12,0)*{\OriSquare{}{}{$f_4$}{}{$f_3$}};
          (13,0)*{\OriSquare{}{}{$f_5$}{}{$f_1$}};
          (14,0)*{\OriSquare{}{}{$a'_{1}$}{}{$c'_{1}$}};
          (15,0)*{\OriSquare{}{}{$b'_{1}$}{}{$a'_{1}$}};
          (16,0)*{\OriSquare{}{}{$c'_{1}$}{}{$b'_{1}$}};
          (17,0)*{\OriSquare{$\cdots$}{}{$\cdots$}{}{$\cdots$}};
          (18,0)*{\OriSquare{$\cdots$}{}{$\cdots$}{}{$\cdots$}};
          (19,0)*{\OriSquare{$\cdots$}{}{$\cdots$}{}{$\cdots$}};
          (20,0)*{\OriSquare{}{}{$a'_{k_2}$}{}{$c'_{k_2}$}};
          (21,0)*{\OriSquare{}{}{$b'_{k_2}$}{}{$a'_{k_2}$}};
          (22,0)*{\OriSquare{}{}{$c'_{k_2}$}{}{$b'_{k_2}$}};
          (23,0)*{\OriSquare{$N$}{}{$d_1$}{}{$d_1$}};
          (24,0)*{\OriSquare{\ldots}{}{}{}{}};       
          (25,0)*{\OriSquare{$N'$}{$x$}{$d_{l}$}{}{$d_{l}$}};
          (-.5,.5)*{\bullet};(1.5,.5)*{\bullet};(2.5,.5)*{\bullet};(4.5,.5)*{\bullet};(5.5,.5)*{\bullet};
          (7.5,.5)*{\bullet};(8.5,.5)*{\bullet};(11.5,.5)*{\bullet};(25.5,.5)*{\bullet};
          (-.5,-.5)*{\bullet}; (.5,-.5)*{\bullet};(2.5,-.5)*{\bullet};(3.5,-.5)*{\bullet};(5.5,-.5)*{\bullet};
          (6.5,-.5)*{\bullet};(8.5,-.5)*{\bullet};(12.5,-.5)*{\bullet};(25.5,-.5)*{\bullet};
          (13.5,.5)*{\circ}; (15.5,.5)*{\circ}; (16.5,.5)*{\circ}; (18.5,.5)*{\circ}; (19.5,.5)*{\circ};
          (21.5,.5)*{\circ}; (22.5,.5)*{\circ};
          (13.5,-.5)*{\circ}; (14.5,-.5)*{\circ}; (16.5,-.5)*{\circ};(17.5,-.5)*{\circ};(19.5,-.5)*{\circ};
          (20.5,-.5)*{\circ};(22.5,-.5)*{\circ}; 
          (9.5,.5)*{\circ}; (10.5,-.5)*{\circ};
        \end{xy}
      \end{xy}

      \vspace*{3mm}
      \caption{The origami $O(\alpha_1,\alpha_2;l)$ from 
        \Cref{example:building_blocks}}
      \label{origami3}
    \end{center}
  \end{figure}
\end{itemize}
\end{ex}


We may now construct one-cylinder origamis in a general stratum $\HHH(\alpha_1,\ldots, \alpha_k)$  
by cutting and pasting the origamis from \Cref{example:building_blocks} as described in the following.
We assume that the numbers $\alpha_1 $, $\ldots$,  $\alpha_k$ are ordered such that the first part
consists of even numbers and the second part of odd numbers.
Recall that $\alpha_1+1$, \ldots, $\alpha_k+1$ are the cycle lengths of the commutator $[\sigma_b^{-1},\sigma_a^{-1}]$.
Since the commutator is an even permutation, the number of odd  $\alpha_i$ is even.

\begin{lem}\label{theorigami} 
Let $\alpha_1 \ldots, \alpha_p$ be even, $\alpha_{p+1}, \ldots, \alpha_{p+2q}$ be odd numbers.
Let further $l$ be a natural number. 
We obtain a one-cylinder origami $O$ in $\HHH(\alpha_1, \ldots, \alpha_{p+2q})$ with 
\[L = \frac{3}{2}(\alpha_1 + \ldots +\alpha_{p+2q}) + p + 3q+l-1\] squares as follows (cf. \Cref{exH2413}). 
If $q \neq 0$,  we take the origamis 
\[O(\alpha_1), \ldots, O(\alpha_p), O(\alpha_{p+1},\alpha_{p+2}), \ldots, 
  O(\alpha_{p+2q-3},\alpha_{p+2q-2}) \mbox{ and } O(\alpha_{p+2q-1},\alpha_{p+2q};l)\] defined in  \Cref{example:building_blocks}. 
We cut them along the left vertical edge of their first square which
is equal to the right vertical edge of their last square. We then glue them in the stated order along these slits.
If $q = 0$, we take the origamis $O(\alpha_1), \ldots, O(\alpha_{p-1}), O(\alpha_{p};l)$ and do the same procedure.\\
This means the origami $O$ is defined by the 
two permutations $(\sigma_a,\sigma_b)$ given as follows: If $q \neq 0$, we have
\begin{equation}\label{thepermutations}
  \begin{array}{lcl}
  \sigma_a &=& (1, \ldots, L), \quad\\
  \sigma_b &=& 
      \begin{array}[t]{lcl}
                \sigmahat_b(\alpha_1) &\circ& \ldots \circ \sigmahat_b(\alpha_p)\circ \sigmahat(\alpha_{p+1},\alpha_{p+2}) 
                  \circ \ldots \circ  \sigmahat(\alpha_{p+2q-3},\alpha_{p+2q-2})\\ 
                                      &\circ& \sigmahat(\alpha_{p+2q-1},\alpha_{p+2q};l)
      \end{array}
  \end{array}
\end{equation}
Here $\sigmahat_b(\alpha_i)$, $\sigmahat_b(\alpha_{i},\alpha_{i+1})$ and $\sigmahat(\alpha_{p+2q-1},\alpha_{p+2q};l)$
are conjugates of $\sigma_b(\alpha_i)$, $\sigma_b(\alpha_{i},\alpha_{i+1})$ and $\sigma(\alpha_{p+2q-1},\alpha_{p+2q};l)$
which shift the labels of $O(\alpha_i)$, $O(\alpha_i,\alpha_{i+1})$ and $O(\alpha_{p+2q-1},\alpha_{p+2q};l)$
by the sum of the lengths of the origamis before them.
More precisely, we define these permutations in the following way. Let $s_i = \frac{3}{2}\alpha_i + 1$ if $i \leq p$
and $s_i = \frac{3}{2}\alpha_i + \frac{3}{2}$ if $p+1 \leq i \leq p+2q-1$. Then $O(\alpha_i)$
is of length $s_i$ for $i \leq p$ and $O(\alpha_i,\alpha_{i+1})$ is of length $s_{i} + s_{i+1}$ for $p+1 \leq i \leq p+2q-3$. 
Define $S_i = \sum_{j = 1}^{i-1}s_j$. Let furthermore $\sh(a): \NN \to \NN$ be the map $n \mapsto n + a$.
Then 
\[\begin{array}{lcl}
 \sigmahat_b(\alpha_i) &=& \sh(S_i)\circ \sigma_b(\alpha_i)\circ \sh(S_i)^{-1},\\ 
 \sigmahat_b(\alpha_{i},\alpha_{i+1}) &=& \sh(S_i)\circ \sigma_b(\alpha_i,\alpha_{i+1}) \circ \sh(S_i)^{-1}, \mbox{ and }\\
  \sigmahat_b(\alpha_{p+2q-1}, \alpha_{p+2q};l) &=&  \sh(S_{p+2q-1})\circ \sigma_b(\alpha_{p+2q-1},\alpha_{p+2q};l) \circ \sh(S_{p+2q-1})^{-1}
  \end{array}
\]
If $q = 0$, we similarly have
\[\sigma_a = (1,\ldots,L) \mbox{ and } \sigma_b =  \sigmahat_b(\alpha_1) \circ \ldots \circ \sigmahat_b(\alpha_{p-1}) 
\circ \sigmahat_b(\alpha_p;l),\] with 
 $\sigmahat_b(\alpha_1)$, \ldots, $\sigmahat_b(\alpha_{p-1})$ and $\sigmahat_b(\alpha_p;l)$ 
defined as conjugates of  $\sigma_b(\alpha_1)$, \ldots, $\sigma_b(\alpha_{p-1})$ and $\sigma_b(\alpha_p;l)$ 
with the suitable shifts similarly as in the case $q \neq 0$.\\
\end{lem}

Figure~\ref{exH2413} shows the origami in $\HHH(2,4,1,3)$ obtained by this construction with
$l = 2$. 

\begin{proof}
Assume first that $q \neq 0$. You can directly check  from the definition of $O$ and \Cref{example:building_blocks} that each 
building block $O(\alpha_i)$ contributes one singularity of order $\alpha_i$ to the surface. 
Furthermore, each $O(\alpha_i,\alpha_{i+1})$ contributes two singularities of order
$\alpha_i$ and $\alpha_{i+1}$. $O(\alpha_{p+2q-1}, \alpha_{p+2q};l)$ also contributes 
two singularities of order $\alpha_{p+2q-1}$ and $\alpha_{p+2q}$. 
Finally, the numbers of squares of the origamis $O(\alpha_1)$, \ldots, 
$O(\alpha_{p+2q-1},\alpha_{p+2q};l)$
add up to the number $L$ of squares of the constructed origami $O$. Thus we obtain:
\[
\begin{array}{lcl}
  L &=& \frac{3}{2}\alpha_1 + 1 + \ldots + \frac{3}{2}\alpha_p + 1 \\[1mm] 
    &&   + \frac32(\alpha_{p+1} + \alpha_{p+2}) + 3 + \ldots + \frac32(\alpha_{p+2q-1} + \alpha_{p+2q}) + 3 + l-1\\[1mm]
    &=& \frac32(\alpha_1 +  \ldots + \alpha_{p+2q}) + p + 3q + l-1 
\end{array}
\]
The proof works similarly if $q = 0$.
\end{proof}

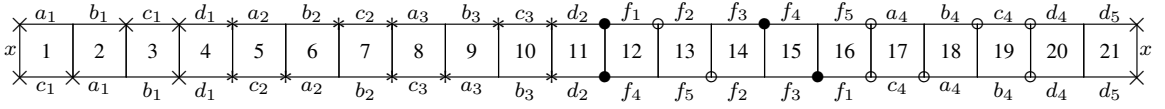
\begin{figure}
  \label{ex3}
    \hspace*{1mm}
    \begin{center}      
      \begin{xy}
        <0.7cm,0cm>:
        (0,0)*{\OriSquare{1}{}{$a_1$}{$x$}{$c_1$}};
        (1,0)*{\OriSquare{2}{}{$b_1$}{}{$a_1$}};
        (2,0)*{\OriSquare{3}{}{$c_1$}{}{$b_1$}};
        (3,0)*{\OriSquare{4}{}{$d_1$}{}{$d_1$}};
        (4,0)*{\OriSquare{5}{}{$a_2$}{}{$c_2$}};
        (5,0)*{\OriSquare{6}{}{$b_2$}{}{$a_2$}};
        (6,0)*{\OriSquare{7}{}{$c_2$}{}{$b_2$}};
        (7,0)*{\OriSquare{8}{}{$a_3$}{}{$c_3$}};
        (8,0)*{\OriSquare{9}{}{$b_3$}{}{$a_3$}};
        (9,0)*{\OriSquare{10}{}{$c_3$}{}{$b_3$}};
        (10,0)*{\OriSquare{11}{}{$d_2$}{}{$d_2$}};
        (11,0)*{\OriSquare{12}{}{$f_1$}{}{$f_4$}};
        (12,0)*{\OriSquare{13}{}{$f_2$}{}{$f_5$}};
        (13,0)*{\OriSquare{14}{}{$f_3$}{}{$f_2$}};
        (14,0)*{\OriSquare{15}{}{$f_4$}{}{$f_3$}};
        (15,0)*{\OriSquare{16}{}{$f_5$}{}{$f_1$}};
        (16,0)*{\OriSquare{17}{}{$a_4$}{}{$c_4$}};
        (17,0)*{\OriSquare{18}{}{$b_4$}{}{$a_4$}};
        (18,0)*{\OriSquare{19}{}{$c_4$}{}{$b_4$}};
        (19,0)*{\OriSquare{20}{}{$d_4$}{}{$d_4$}};
        (20,0)*{\OriSquare{21}{$x$}{$d_5$}{}{$d_5$}};
        (-.5,.55)*{\times}; (1.5,.55)*{\times}; (2.5,.55)*{\times};  (20.5,.55)*{\times};   
        (-.5,-.44)*{\times};  (.5,-.44)*{\times};  (2.5,-.44)*{\times}; (20.5,-.44)*{\times};
        (3.5,.56)*{\ast}; (5.5,.56)*{\ast}; (6.5,.56)*{\ast};
        (8.5,.56)*{\ast};(9.5,.56)*{\ast};
        (3.5,-.44)*{\ast}; (4.5,-.44)*{\ast}; (6.5,-.44)*{\ast}; 
        (7.5,-.44)*{\ast}; (9.5,-.44)*{\ast}; 
        (10.5,.56)*{\bullet};(13.5,.56)*{\bullet};
        (10.5,-.44)*{\bullet}; (14.5,-.44)*{\bullet};
        (11.5,.56)*{\circ}; (15.5,.56)*{\circ}; (17.5,.56)*{\circ}; (18.5,.56)*{\circ};
        (12.5,-.44)*{\circ}; (15.5,-.44)*{\circ};  (16.5,-.44)*{\circ}; (18.5,-.44)*{\circ}; 
      \end{xy}
      \caption{The origami in $\HHH(2,4,1,3)$ with $l = 2$ 
        obtained from the construction 
        in \Cref{theorigami}}\label{exH2413}
    \end{center}
\end{figure}

In the following we consider cylinder decompositions in different directions of the origamis constructed in 
\Cref{theorigami}. Based on this we obtain  parabolic elements in the Veech groups of these origamis. 

\begin{lem}\label{lem:parabolics}
Let $\Gamma$ be the Veech group of the origami $O = O(l)$ 
with $L = \frac{3}{2}(\alpha_1 + \ldots +\alpha_{p+2q}) + p + 3q+l-1$  squares
constructed in \Cref{theorigami}.
Then $\Gamma$ contains the following parabolic matrices:
\[T^L, T'^{15}, \mbox{ and } T''^{2(L-4q)}, \mbox{ with } T \mbox{ and } T' \mbox{ defined in (\ref{defT})} \mbox{ and }T'' = T'TT'^{-1}\]
\end{lem}

\begin{proof}
It follows from  its definition that $O$ consists of one horizontal
cylinder which has length $L$ and height 1. Thus the Veech group 
contains the matrix $T^L$. Furthermore,  since all cycles of
$\sigma_b$ are of length 1, 3 or 5, we have that  $O$ decomposes into 
vertical cylinders of height 1 and length 1, 3 or 5. Hence $T'^{15}$
is contained in $\Gamma$. Finally, the origami $T'^{-1}\cdot O$
is given by the two permutations $(\sigma_b\sigma_a,\sigma_b)$.
We will show below that $\sigma_b\sigma_a$ consists of one cycle
of length $L-4q$ and further cycles of length 2.
Hence $T'^{-1}\cdot O$ composes into horizontal cylinders
of length $L-4q$ and of length 2. Therefore
$T^{2(L-4q)} \in \Gamma(T'^{-1}O) = T'^{-1}\Gamma T'$ and thus
$T'T^{2(L-4q)}T'^{-1} = T''^{2(L-4q)}\in \Gamma$. This finishes the claim.\\
Let us now show that the permutation $\sigma_b\sigma_a$
is of the desired form.  We assume that $q \neq 0$. The case $q = 0$ 
works in the same way. Recall that $O$ consists of the origamis $O(\alpha_1)$, \ldots, $O(\alpha_p)$, 
$O(\alpha_{p+1},\alpha_{p+2})$, \ldots, $O(\alpha_{p+2q-3},\alpha_{p+2q-2})$, $O(\alpha_{p+2q-1},\alpha_{p+2q};l)$
which are glued in a row along slits. We label the squares of $O$ from left to right by
$1, \ldots, L \in \ZZ/L\ZZ$. Let us consider how the permutation $\sigma_b\sigma_a$
acts on the labels of the squares.\\
Recall the definition of $S_i$ and $s_i$ in \Cref{theorigami}.
The origamis $O(\alpha_i)$ are then of length $s_i$ and the origamis 
$O(\alpha_i,\alpha_{i+1})$ are of length $s_i + s_{i+1}$.
Let us consider the  squares belonging to 
the origami $O(\alpha_i)$ ($i \in \{1, \ldots, p\}$). 
The first square of the origami $O(\alpha_i)$ is labelled
by $S_i+1$ and the last one is labelled by $S_i + s_i$. 
Observe (cf.~\Cref{origami1}) that the permutation $\sigma_b\sigma_a$ acts in the following way:  
\[
  \begin{array}{l}
  S_i \mapsto  S_i + 2 \mapsto S_i + 1 \mapsto S_i + 3 \mapsto S_i + 5 \mapsto S_i + 4 \mapsto S_i + 6 \mapsto  \ldots\\ 
  \hspace*{1cm}\mapsto S_i + s_i - 2 \mapsto S_i + s_i - 3 \mapsto S_i + s_i - 1 \mapsto S_i + s_i = S_{i+1} 
  \end{array}
\]
In particular all squares of the origamis $O(\alpha_1)$, \ldots, $O(\alpha_{p})$,
i.e. all squares labelled by $1$, $2$, \ldots, $S_{p+1}$, lie in the same orbit.\\
Let us now consider the origamis $O(\alpha_{i},\alpha_{i+1})$ ($i-p$ odd, $1 \leq i \leq 2q-3$).
The first square of $O(\alpha_{i},\alpha_{i+1})$ is labelled by $S_i+1$, the last one by $S_i + s_i + s_{i+1}$. 
Observe that $\sigma_b\sigma_a$ acts in the following way (cf.~\Cref{origami3}):
\[
  \begin{array}{ll}
    \multicolumn{2}{l}{\mbox{Denote } k_i = \frac{\alpha_i-1}{2} \mbox{ and }  k_{i+1} = \frac{\alpha_{i+1}-1}{2}.}\\[1mm]
    S_i &\mapsto   S_i + 2 \mapsto S_i + 1 \mapsto S_i + 3 \mapsto  S_i + 5 \mapsto S_i + 4 \mapsto S_i + 6 \mapsto \ldots \\ 
        &\mapsto S_i + 3k_{i} - 1 \mapsto S_i + 3k_{i} - 2 \mapsto  S_i + 3k_{i} \mapsto S_i + 3k_{i} + 5\\ 
        &\mapsto S_i + 3k_{i} +7  \mapsto S_i + 3k_i+6  \mapsto S_i + 3k_i+8  \mapsto \ldots  \\
        &\mapsto S_i + 3(k_i + k_{i+1})+4 \mapsto S_i + 3(k_i + k_{i+1}) + 3 \mapsto S_i + 3(k_i + k_{i+1}) + 5 \\
        &\mapsto S_i + 3(k_i + k_{i+1}) + 6\\ 
  \end{array}
\]

The remaining squares of $O(\alpha_i,\alpha_{i+1})$ which do not belong to this orbit
are $S_i+3k_i + 1$, $S_i+3k_i + 2$,  $S_i+3k_i + 3$ and   $S_i+3k_i + 4$.
They form two cycles $(S_i+3k_i + 1,  S_i+3k_i + 3)$ and  $(S_i+3k_i + 2,  S_i+3k_i + 4)$
of length two.\\

Similarly, the permutation $\sigma_b\sigma_a$ acts on the squares of the origami $O(\alpha_{p+2q-1},\alpha_{p + 2q};l)$ by:
\[
\begin{array}[t]{ll}
  \multicolumn{2}{l}{\mbox{Denote $i = p+2q-1$.}}\\[1mm]
  S_i & \mapsto   S_i + 2 \mapsto S_i + 1 \mapsto S_i + 3 \mapsto  S_i + 5 \mapsto S_i + 4 \mapsto S_i + 6 \mapsto \ldots \\ 
      &  \mapsto S_i + 3k_{i} - 1 \mapsto S_i + 3k_{i} - 2 \mapsto  S_i + 3k_{i} \mapsto S_i + 3k_{i} + 5\\ 
      &  \mapsto S_i + 3k_{i} +7  \mapsto S_i + 3k_i+6  \mapsto S_i + 3k_i+8  \mapsto \ldots  \\
      &  \mapsto S_i + 3(k_i + k_{i+1})+4 \mapsto S_i + 3(k_i + k_{i+1}) + 3 \mapsto S_i + 3(k_i + k_{i+1}) + 5\\
      &  \mapsto S_i + 3(k_i + k_{i+1}) + 6 \mapsto \ldots \mapsto  S_i + 3(k_i + k_{i+1}) + 5 + l\\[2mm]
  \multicolumn{2}{l}{\mbox{and by two cycles } (S_{i}+3k_{i} + 1,  S_{i}+3k_{i} + 3) \mbox{ and } (S_{i}+3k_{i} + 2,  S_{i}+3k_{i} + 4)}
  \end{array}
\]

Altogether, we obtain for the permutation $\sigma_b\sigma_a$ 
one long cycle containing all squares except the squares $S_i+3k_i + 1$, $S_i+3k_i + 2$,  $S_i+3k_i + 3$ and  $S_i+3k_i + 4$
with $i-p$ odd and $p+1 \leq i \leq p+2q$. This circle has length $L - 4q$. Furthermore, we obtain $2q$ cycles of length 2.
Hence $\sigma_b\sigma_a$ has the form which we claimed.
\end{proof}

We are now able to obtain explicit origamis in each stratum 
whose Veech groups are totally non congruence groups.

\begin{prop}\label{prop:conditions}
  Let $\alpha_1 \ldots, \alpha_p$ be even, $\alpha_{p+1}, \ldots, \alpha_{p+2q}$ be odd numbers.
  Recall that in \Cref{theorigami} we constructed an origami $O$ in 
  $\HHH(\alpha_1, \ldots, \alpha_{p+2q})$ 
  with $L$ squares, where
  \[L = \frac{3}{2}(\alpha_1 + \ldots +\alpha_{p+2q}) + p + 3q+l-1.\]
  Choose $l \in \NN$ such that:
  \begin{enumerate}
  \item[i)]
    $\gcd(L,30q) = 1$.
  \item[ii)]
    $3$ and $5$ do not divide $L-4q$.
  \end{enumerate}
Then the Veech group $\Gamma = \Gamma(O)$ of $O$ is a totally non congruence group. 
\end{prop}

\begin{proof}
We know from \Cref{lem:parabolics} that the matrices
\[ T^L, T'^{15} \mbox{ and  }   T''^{2(L-4q)} \mbox{ with }   
   T'' = T'TT'^{-1} =  \bpm 0 & 1\\ -1 & 2 \epm \]
are contained in $\Gamma$. We apply \Cref{new-criterion}. Observe firstly
that each pair $(A_1,A_2)$ of two matrices in $\{T, T',T''\}$
satisfies property A) in  \Cref{new-criterion} for any prime $p$.
We distinguish now three
cases. Suppose as first case that $p$  is neither a divisor
of $L$ nor of $15$.  
Then we choose 
$A_1 = T$, $A_2 = T'$, $m_1 = L$ and $m_2 = 15$. By the assumption on $p$
we have that $p$ does neither divide $m_1$ nor $m_2$.
As second case we consider that $p$ divides $L$. Then we
choose $A_1 = T'$, $A_2 = T''$, $m_1 = 15$ and $m_2 = 2(L-4q)$.
Now, $p$ does not divide $m_1$ by i). Furthermore, it follows from i)
that $p$ does not divide $4q$. Thus since it is is a divisor of $L$,
it does not divide  $m_2 = L-4q$.
In the remaining case, namely $p = 3$ or $p = 5$, we choose    
$A_1 = T$, $A_2 = T''$, $m_1 = L$ and $m_2 = 2(L-4q)$. In this case
$p$ does neither divide $m_1$ (by i)) nor $m_2$ (by ii)).
Hence, in all three cases we obtain that also property B) in  \Cref{new-criterion}
holds. This finishes the proof.
\end{proof}

In particular, \Cref{prop:conditions} defines in each stratum an infinite family of origamis.
    
\begin{thm}\label{thm:in_all_strata}
Every stratum contains an infinite family of origamis
whose Veech groups are totally non-congruence groups.
\end{thm}

\begin{proof}
The theorem directly follows from \Cref{prop:conditions}.
Namely, we can choose $l$ for example 
such that $L$ is a prime with $L > 4q$ which satisfies the following conditions:
\[
  \begin{array}{lcl}
    L &\equiv& \left\{
          \begin{array}{l} 
            4q+1 \mod 3, \mbox{ if } 3 \mbox{ does not divide } 4q+1\\
            4q+2 \mod 3, \mbox{ elsewise } 
          \end{array} \right . \\[1mm]
    L &\equiv& \left\{
          \begin{array}{l} 
            4q+1 \mod 5, \mbox{ if } 5 \mbox{ does not divide } 4q+1\\
            4q+2 \mod 5, \mbox{ elsewise } 
          \end{array} \right . 
  \end{array} 
\]
By Dirichlet's theorem on arithmetic progressions there are infinitely many
primes which satisfy these conditions. 
\end{proof}

\bibliographystyle{amsalpha}
\bibliography{cg}

\end{document}